\documentclass[10pt]{article}
\usepackage{amsmath, amssymb, amsfonts, amsthm, epsfig, fdsymbol, float, graphicx, sectsty}
\usepackage[all]{xy}
\title{On the spectrum of residual finiteness growth functions}
\author{Henry Bradford}

\newtheorem{thm}{Theorem}[section]
\newtheorem{lem}[thm]{Lemma}
\newtheorem{propn}[thm]{Proposition}
\newtheorem{coroll}[thm]{Corollary}
\newtheorem{defn}[thm]{Definition}
\newtheorem{ex}[thm]{Example}

\newtheorem{rmrk}[thm]{Remark}
\newtheorem{qu}[thm]{Question}

\DeclareMathOperator{\Alt}{Alt}

\DeclareMathOperator{\PSL}{PSL}

\DeclareMathOperator{\SL}{SL}

\DeclareMathOperator{\Sym}{Sym}

\begin{document}

\maketitle

\begin{abstract}
In \cite{BoRaSewa} Bou-Rabee and Seward 
constructed examples of 
finitely generated residually finite groups $G$
whose residual finiteness growth function 
$\mathcal{F}_G$ can 
be at least as fast as any prescribed function. 
In this note we describe a modified version of 
their construction, which allows us to 
give a complementary upper bound on $\mathcal{F}_G$. 
As such, every nondecreasing function at least 
$\exp ( n \log (n)^2 \log \log (n)^{1+\epsilon} )$ is close to the residual finiteness 
growth function of some finitely generated group. 
We also have similar result for the 
\emph{full} residual finiteness growth function 
and for the divisibility function. 
\end{abstract}

\section{Introduction}

A group $G$ is \emph{residually finite} if 
every nontrivial element can be distinguished 
from the identity in some finite quotient of $G$. 
In \cite{BouRab} Bou-Rabee pioneered the 
study of an effective version of residual finiteness, 
valid when $G$ is generated by a finite set $S$, 
and measured by the 
\emph{residual finiteness growth function} 
$\mathcal{F}_G ^S : 
\mathbb{N} \rightarrow \mathbb{N}$ 
of $G$. 
The project of estimating $\mathcal{F}_G ^S$ 
for groups $G$ of particular interest 
has been taken up by many authors; 
see \cite{BoRaCheTim} Table 1 
for a summary of some of the known estimates. 
One may also pose the ``inverse'' problem 
and ask which functions can arise as the 
residual finiteness growth of some group $G$. 
For instance, one might ask whether 
there is a universal upper bound on the 
rate at which $\mathcal{F}_G ^S$ can grow. 
Bou-Rabee and Seward \cite{BoRaSewa} 
answer this last question in the negative: 
they construct, for any increasing function $F$, 
a finitely generated residually finite group $G$ 
such that $\mathcal{F}_G ^S$ dominates $F$.  
Our goal is to strengthen their result, 
by showing that for all sufficiently quickly growing $F$, we can arrange that 
$\mathcal{F}_G ^S$ is \emph{equivalent} to 
$F$, in an appropriate sense. 

\begin{thm}[Theorem \ref{MainThmRFGrestated}] \label{MainThmRFG}
Let $F : \mathbb{N} \rightarrow \mathbb{N}$ 
be a nondecreasing function such that: 
\begin{itemize}
\item[(a)] There exist $c ,\epsilon > 0$ 
such that for all $n \in \mathbb{N}$, 
\begin{equation*}
F(n) \geq \exp \big( c n \log (n)^2 \log \log (n)^{1+\epsilon} \big); 
\end{equation*}

\item[(b)] There exist $C_1, C_2 > 1$ 
such that for all $n \in \mathbb{N}$, 
\begin{equation*}
F(n)^{C_1} \leq F(C_2 n). 
\end{equation*}
\end{itemize}

Then there exists a residually finite group 
$G$, generated be a finite set $S$, 
and $C_1 ',C_2 ' >0$ such that for all $n \in \mathbb{N}$, 
\begin{equation} \label{MainThmConclIneq}
F(C_1'n) \leq \mathcal{F}_G ^S (n) \leq F(C_2'n). 
\end{equation}
\end{thm}

In the terminology of \cite{DeLaHar}, 
inequalities of the form (\ref{MainThmConclIneq}) 
constitute a \emph{strong equivalence} 
between the functions $F$ and $\mathcal{F}_G ^S$. 
Absent a condition of the form (b) 
on $F$, we can still show an equivalence 
between the growth of $\mathcal{F}_{G}$ 
and $F$, in  a slightly weaker sense 
(see Theorem \ref{MainTechRFG} below 
for the general statement). 



There are many equivalent ways to define 
residual finiteness of a group, 
some of which lead to distinct growth functions. 
Besides $\mathcal{F}_G ^S$ , 
we study the \emph{full residual finiteness growth 
function} (also called \emph{residual girth}) 
$\mathcal{R}_G ^S$ introduced by Bou-Rabee and 
McReynolds \cite{BoRaMcR} 
and the \emph{divisibility function} 
$\mathcal{D}_G ^S$ 
(also called the \emph{non-normal residual finiteness growth function}). 
For any finitely generated group $G$, 
we have $\mathcal{D}_G ^S (n) \leq \mathcal{F}_G ^S (n) \leq \mathcal{R}_G ^S (n)$. 
For a particular $G$, 
it might happen that 
$\mathcal{F}_G ^S$ and $\mathcal{R}_G ^S$ are equivalent, 
or that $\mathcal{R}_G ^S$ grows much faster 
than $\mathcal{F}_G ^S$ 
(and similarly for the inequality between 
$\mathcal{D}_G ^S$ and $\mathcal{F}_G ^S$). 
To take one example, $\mathcal{F}_G ^S$ can grow 
polynomially, while $\mathcal{R}_G ^S$ 
grows exponentially, as in the case $G = \SL_d (\mathbb{Z})$, 
for $d \geq 3$. 
In our case, the same construction that 
yields Theorem \ref{MainThmRFG} 
also allows us to prove the following. 

\begin{thm}[Theorem \ref{MainThmFullRFGrestated}] \label{MainThmFullRFG}
Let $F : \mathbb{N} \rightarrow \mathbb{N}$ 
be a nondecreasing function 
satisfying: 
\begin{itemize}
\item[(b')] There exists $C > 1$ 
such that for all $n \in \mathbb{N}$, 
$F(n)^n \leq F(Cn)$. 
\end{itemize}
Then there exists a residually finite group 
$G$, generated be a finite set $S$, 
and $C_1 ',C_2 ' >0$ 
such that for all $n \in \mathbb{N}$, 
\begin{equation} 
F(C_1'n)  \leq \mathcal{F}_G ^S (n) \leq \mathcal{R}_G ^S (n) \leq F(C_2'n). 
\end{equation}
\end{thm}

\begin{thm}  \label{MainThmDiv}
Let $f : \mathbb{N} \rightarrow \mathbb{N}$ 
be a nondecreasing function such that: 
\begin{itemize}
\item[(a'')] There exist $c ,\epsilon > 0$ 
such that for all $n \in \mathbb{N}$, 
\begin{equation*}
f(n) \geq c n \log (n) \log \log (n)^{1+\epsilon}; 
\end{equation*}

\item[(b'')] There exists $C_1 , C_2 > 1$ 
such that for all $n \in \mathbb{N}$, 
$C_1 f(n) \leq f(C_2 n)$. 
\end{itemize}
Then there exists a residually finite group 
$G$, generated be a finite set $S$, 
and $C_1 '',C_2 '' >0$ 
such that for all $n \in \mathbb{N}$, 
\begin{equation} 
f(C_1''n)  \leq \mathcal{D}_G ^S (n) \leq f(C_2''n). 
\end{equation}
\end{thm}

Again, we also have results under conditions 
weaker than (b') or (b''); 
see Theorems \ref{MainTechFRFG} and \ref{MainTechDiv} 
for the most general statements. 
Note that condition (b') is stronger 
than condition (b) of Theorem \ref{MainThmRFG}. 
Moreover, an easy induction shows that condition (b') 
implies a bound $F(n) \geq \exp(c n^{\log(n)})$, 
so that a function satisfying (b') 
also satisfies condition (a) of Theorem \ref{MainThmRFG}. 
Our methods are not applicable to functions 
$F$ growing too slowly. 
For instance it remains a beguiling open question 
whether there exists a finitely 
generated residually finite group for which 
$\mathcal{F}_G ^S$ or $\mathcal{R}_G ^S$ 
is superpolynomial but subexponential. 

\subsection*{Notation}

If $f : \mathbb{N} \rightarrow \mathbb{N}$ 
is a function and $x \in (0,\infty)$ 
is not an integer, 
then by $f(x)$ we shall mean $f(\lceil x \rceil)$. 
For $G$ a group and $g,h \in G$, 
let $[g,h] = g^{-1} h^{-1} gh$ denote 
the commutator of $g$ and $h$ in $G$. 
By $F(x,y)$ we shall denote the free group 
of rank two with free basis $\lbrace x,y \rbrace$. 
For $w \in F(x,y)$ a reduced word, 
$G$ a group and $g,h \in G$, 
we shall denote by $w(g,h)$ the image of $w$ under the 
(unique) homomorphism $F(x,y) \rightarrow G$ 
sending $x,y$ to $g,h$, respectively. 
In this paper all group actions are on the left. 

\section{Preliminaries}

\subsection{Functional inequalities}

\begin{defn}
Given nondecreasing 
functions $f,g:\mathbb{N}\rightarrow\mathbb{N}$, 
we say that $f$ \emph{strongly dominates} $g$ 
if there exists $C>0$ such that for 
all $n \in \mathbb{N}$, 
$g(n) \leq f(Cn)$. 
We say that $f$ and $g$ are \emph{strongly equivalent} 
and write $f \approx^s g$ if each strongly 
dominates the other. 
\end{defn}

\begin{rmrk} \label{SuffLargeRmrk}
If $f$ is unbounded, 
then to prove that $f$ strongly dominates $g$, 
it suffices to verify an inequality 
of the form $g(n) \leq f(Cn)$ for all $n$ 
sufficiently large. 
\end{rmrk}

\begin{lem} \label{FactorialInvLem}
Let $F : \mathbb{N} \rightarrow \mathbb{N}$ 
be a nondecreasing function satisfying 
hypothesis (a) of Theorem \ref{MainThmRFG} 
(with respect to $c,\epsilon > 0$). 
Given $C_1 > 0$, there exists $C_2 > 0$ 
such that the function $f(n) = \lceil \log F(n+C_2) / \log\log F(n+C_2) \rceil$ satisfies: 
\begin{equation}
f(n) \geq C_1 n \log (n) \big( \log \log(n) \big)^{1+\frac{\epsilon}{2}} + C_1
\end{equation}
for all $n \in \mathbb{N}$. 
\end{lem}

\begin{proof}
Since $\log(x)/\log\log(x)$ is nondecreasing 
for $x$ sufficiently large, we have for all 
$n \in \mathbb{N}$ sufficiently large 
(depending on $\epsilon$, $c$ and $C_1$): 
\begin{align*}
\frac{\log F(n)}{\log \log F(n)} & \geq 
\frac{c n \log (n)^2 \log \log (n)^{1+\epsilon}}{\log (n) + 2 \log\log (n)+(1+\epsilon)\log\log\log(n) + \log(c)} \\
& \geq c n \log (n)^2 \log \log (n)^{1+\epsilon}/2 \\
& \geq C_1 n \log (n) \log \log(n) ^{1+\frac{\epsilon}{2}} + C_1. 
\end{align*}
The required conclusion follows for any $C_2$ 
sufficiently large, 
since $n \log (n) \log \log(n) ^{1+\frac{\epsilon}{2}}$ 
is nondecreasing. 
\end{proof}

\begin{lem} \label{FactorialBdLem}
Let $G : \mathbb{N} \rightarrow \mathbb{N}$ 
be nondecreasing unbounded and let $K \in \mathbb{N}$. 
Let $g(n) = \lceil \log G(n) / \log\log G(n) \rceil$. 
Then: 
\begin{itemize}
\item[(a)] $\log((K g(n))!) = K \log G(n) 
+ O_K \big( \frac{\log G(n)\log\log\log G(n)}{\log\log G(n)} \big)$

\item[(b)]  $\log((K n g(n))!) \leq K n \log G(n) 
+ O_K \big( \frac{n \log G(n) \log(n)}{\log\log G(n)} \big)$

\end{itemize}
\end{lem}

Recall that $O_K(x)$ refers to a quantity 
bounded in absolute value by a multiple of $x$ 
which depends only on $K$. 

\begin{proof}
We assume throughout the proof, as we may, 
that $n$ is sufficiently large. 
We shall recall (a weak form of) Stirling's approximation: 
\begin{center}
$\exp(n\log(n)-n) \leq n! \leq \exp(n\log(n))$. 
\end{center}
Note also that: 
\begin{align*}
\log G(n) / \log\log G(n) \leq g(n) 
& \leq \log G(n) / \log\log G(n) + 1 \\ & \leq 2 \log G(n) / \log\log G(n)
\end{align*}
and thus $\log g(n) \leq \log\log G(n)$. 
We therefore have: 
\begin{align*}
\log((K g(n))!) & \leq Kg(n)\log g(n) + (K \log K) g(n) \\
& \leq K \log G(n) + K \log\log G(n) 
+ 2K\log K \frac{\log G(n)}{\log\log G(n)}
\end{align*}
while: 
\begin{align*}
\log((K g(n))!) & \geq Kg(n)\log g(n) + (K \log K-K )g(n) \\
& \geq K \big( \frac{\log G(n)}{\log\log G(n)} \big)
(\log\log G(n) - \log\log\log G(n))
+ O_K \big( \frac{\log G(n)}{\log\log G(n)} \big)
\end{align*}
yielding (a). 
One may prove the bounds in (b) in much the same way. 
\end{proof}

\subsection{Residual finiteness growth}

Let $G$ be a group, generated by a finite set $S$. 
We denote by $B_S (n)$ the closed ball around $e$ 
in the word-metric induced on $G$ by $S$, 
so that: 
\begin{equation*}
B_S (n) = \lbrace s_1 \cdots s_m 
: 0\leq m\leq n, s_i \in S \cup S^{-1} \rbrace\text{.}
\end{equation*}
Recall that $G$ is \emph{residually finite} if, 
for any nontrivial element $g$ of $G$, 
there exists a finite-index subgroup $H$ of $G$ 
such that $g \notin H$. 
Since every finite-index subgroup contains a finite-index normal subgroup, 
the class of groups captured by this definition is unchanged
if we require $H$ to be normal. 
The minimal index of a finite-index 
(respectively finite-index normal) subgroup of $G$ which does not contain 
the element $g$ is the \emph{depth} 
(respectively \emph{normal depth}) of $g$ in $G$. 
Following \cite{BouRab}, 
the \emph{residual finiteness growth function} 
$\mathcal{F}_{G} ^S : \mathbb{N} \rightarrow \mathbb{N}$ of $G$ 
(with respect to $S$) 
is defined such that $\mathcal{F}_{G} ^S (n)$ 
is the maximal normal depth occurring among the nontrivial 
elements of $B_S(n)$. 
Similarly, the \emph{divisibility function} 
$\mathcal{D}_{G} ^S : \mathbb{N} \rightarrow \mathbb{N}$ of $G$ 
is defined such that $\mathcal{D}_{G} ^S (n)$ 
is the maximal depth occurring among the nontrivial 
elements of $B_S(n)$. 

It is easy to see that if $G$ is residually finite, 
then it satisfies the a priori stronger condition 
that for any finite subset $F$ of $G$, 
there exists a finite-index subgroup $H$ 
such that the coset map $\pi_H :G \rightarrow G/H$ 
(given by $\pi_H(g) = gH$) restricts to an injection on $F$. 
The \emph{injectivity radius} of the subgroup $H$ 
is the maximal $n$ for which the restriction 
of $\pi_H$ to $B_S(n)$ is injective 
(or $+\infty$ if the set of such $n$ is unbounded, 
though note that this only occurs if $G$ itself is finite). 
Following \cite{BoRaMcR}, we define 
the \emph{full residual finiteness growth function} (or \emph{residual girth}) 
$\mathcal{R}_{G} ^S : \mathbb{N} \rightarrow \mathbb{N}$
of $G$ (with respect to $S$) 
such that $\mathcal{R}_{G} ^S (n)$ 
is the minimal index of a normal subgroup of injectivity radius at least $n$. 
Similarly, the \emph{systolic growth function} of $G$ 
is the function $\Sigma_{G} ^S : \mathbb{N} \rightarrow \mathbb{N}$ 
defined such that $\Sigma_{G} ^S (n)$ is the 
minimal index of a subgroup of injectivity radius at least $n$. 
We have the following easy inequalities between the various growth functions
(note however that in many cases these bounds are far from sharp). 

\begin{lem} \label{TrivialRFGrowthIneq}
Let $G$ be a residually finite group, generated by a finite set $S$, 
and let $n \in \mathbb{N}$. 
\begin{itemize}
\item[(i)] $\mathcal{D}_{G} ^S (n) \leq \mathcal{F}_{G} ^S (n) \leq \mathcal{R}_{G} ^S (n)$; 

\item[(ii)] $\mathcal{D}_{G} ^S (n) \leq \Sigma_{G} ^S (n) \leq \mathcal{R}_{G} ^S (n)$; 

\item[(iii)] $\mathcal{F}_{G} ^S (n) \leq \mathcal{D}_{G} ^S (n)^{\mathcal{D}_{G} ^S (n)}$ and $\mathcal{R}_{G} ^S (n) \leq \Sigma_{G} ^S (n)^{\Sigma_{G} ^S (n)}$; 

\item[(iv)] $\mathcal{R}_{G} ^S (n) 
\leq \mathcal{F}_{G} ^S (2n) ^{\lvert B_S (n) \rvert^2}$ 
and $\Sigma_{G} ^S (n) 
\leq \mathcal{D}_{G} ^S (2n) ^{\lvert B_S (n) \rvert^2}$. 

\end{itemize}
\end{lem}

\begin{proof}
If $H \leq G$ has index $n$ then kernel of the action of $G$ 
on the cosets of $H$ is a normal subgroup of index at most $n! \leq n^n$, 
contained in $H$. Items (i)-(iii) are immediate. 
For (iv), note that if $g_1 , g_2 \in G$ satisfy $g_1 ^{-1} g_2 \notin H$, 
then $g_1$ and $g_2$ lie in different cosets of $H$. 
Intersecting such subgroups $H$ corresponding to all pairs of elements 
$g_1 , g_2 \in B_S(n)$ yields the desired bounds. 
\end{proof}

For given $G$, the functions 
$\mathcal{D}_{G} ^S$, $\mathcal{F}_{G} ^S$, $\mathcal{R}_{G} ^S$ 
and $\Sigma_{G} ^S$ depend on the choice of the finite 
generating set $S$ but, since a change of finite generating set 
changes the word metric by at most a multiplicative constant, 
the dependence is slight 
(see for instance Lemma 2.1 and 2.2 of \cite{BoRaMcR}). 

\begin{propn}
Let $G$ be a finitely generated residually 
finite group, and let $S,T\subseteq G$ 
be finite generating sets. 
Let $X \in \lbrace \mathcal{F},\mathcal{R},\mathcal{D},\Sigma \rbrace$. 
Then $X_{G} ^S\approx^s X_{G} ^T$. 
\end{propn}

\subsection{Number-theoretic observations}

\begin{thm}[Bertrand's Postulate]
For any $n \in \mathbb{N}_{\geq 2}$, 
there exists a prime number $p$ 
such that $n < p < 2n$. 
\end{thm}

\begin{coroll} \label{PrimeFnCoroll}
For any nondecreasing function 
$f : \mathbb{N}\rightarrow \mathbb{N}_{\geq 3}$ 
there exists a nondecreasing function 
$d : \mathbb{N}\rightarrow \mathbb{N}$ such that
for all $n \in \mathbb{N}$, $d(n)$ is 
a prime number at least $5$, satisfying 
$f(n) \leq d(n) \leq 2 f(n)$. 
\end{coroll}

\begin{propn} \label{BuildingrProp}
For all $\epsilon > 0$ 
there exists $C_0$ such that for all $C \geq C_0$, 
the following holds 
Let $d , q : \mathbb{N}\rightarrow \mathbb{N}$ be 
nondecreasing functions such that, 
for all $n \in \mathbb{N}$: 
\begin{itemize}
\item[(i)] $d(n)$ is an odd prime number; 

\item[(ii)] $d(n) \geq C n \log (n) \big( \log \log(n) \big)^{1+\frac{\epsilon}{2}} + C$; 

\item[(iii)] $n \leq q(n) \leq d(n)/4$. 

\end{itemize}
Then there exists a function 
$r : \mathbb{N}\rightarrow \mathbb{N}$ satisfying: 
\begin{itemize}
\item[(a)] For all $n \in \mathbb{N}$, 
$q(n) < r(n) < q(n) +  17 n$ and $r(n) < d(n)/3$; 

\item[(b)] For all $l,m\in \mathbb{N}$, 
if $l \neq m$ then 
$r(l) \nequiv \pm r(m),\pm 2 r(m) \mod d(m)$. 

\end{itemize}
\end{propn}

\begin{proof}
First, taking $C_0$ sufficiently large, 
we may assume that: 
\begin{equation*}
d(n) \geq 16 n \text{ for all }n \in \mathbb{N}\text{.}
\end{equation*}
Assume that 
suitable $r(1) , \ldots , r(n-1)$ have 
already been produced. 
We exhibit a positive integer $r(n)$ 
satisfying (a), and such that (b) holds for 
all $1 \leq l,m \leq n$ 
(for the base case, we may simply take $r(1)=2$). 

Let $M \in \mathbb{N}$ (to be chosen). 
We shall seek $r(n)$ in the interval 
$\mathbb{Z} \cap [q(n)+1,q(n)+M]$. 
The series: 
\begin{equation*}
\sum_{m=1} ^{\infty} \frac{1}{d(n)}
\end{equation*}
converges, by hypothesis (ii), 
to a limit $L > 0$. 
Moreover by taking $C_0$ to be sufficiently large, 
we can assume $L < 1/16$. 

For each $1 \leq m \leq n-1$ 
there are at most $4M/d(m) + 4$ ``bad'' integers 
$q(n)+1 \leq k \leq q(n)+M$ for which 
$k \equiv \pm r(m) \text{ or } \pm 2r(m) \mod d(m)$, 
and at most $4M/d(n) + 4$ ``bad'' integers 
$q(n)+1 \leq k \leq q(n)+M$ for which 
$r(m) \equiv \pm k \text{ or } \pm 2 k \mod d(n)$. 
Taking a union over all such $m$, 
we have at most: 
\begin{equation*}
8(n-1) + \frac{4Mn}{d(n)} + 4M \sum_{m=1} ^{n-1} \frac{1}{d(m)} < 8n + \frac{M}{2}
\end{equation*}
``bad'' integers in total in our range. 
Thus for any $M \geq 16n$, we can choose $r(n) \in \mathbb{N}$ with 
$q(n)+1 \leq r(n) \leq q(n)+M$, 
such that $r(n)$ is not ``bad''. 
Thus (b) is indeed satisfied for $1 \leq l,m \leq n$. 
In particular, we may take $M < 17n$ 
and choose $q(n) < r(n) < q(n) + 17n$. 
Choosing $C_0$ sufficiently large, 
this also guarantees $r(n) < d(n)/3$. 
\end{proof}

\begin{rmrk}
\normalfont
For the sake of the results to be proven in this Note, 
it shall be sufficient henceforth 
to take $q(n)=n$ for all $n$. 
In future applications, to be explored elsewhere, 
it shall be convenient to have the flexibility 
provided by other choices for $q$. 
\end{rmrk}

\subsection{Generalized B.H. Neumann groups} \label{NeumannSubsect}

Let $d,r_1,r_2 : \mathbb{N} \rightarrow \mathbb{N}$ 
be functions, 
such that for all $n \in \mathbb{N}$, 
$d(n)$ is an odd number at least $5$, 
and $r_1(n) + r_2(n) \leq d(n)-1$. 
Let $\alpha_n = \big(1 \; 2 \; \cdots \; d(n)\big) , 
\beta_n = \big(1 \; (1+r_1(n)) \; (1+r_1(n)+r_2(n))\big) \in \Alt (d(n))$. 
Under mild conditions, 
the permutations $\alpha_n$ and $\beta_n$ generate 
$\Alt (d(n))$. For instance, we have the following, 
which shall suffice for our purposes here. 

\begin{lem} \label{GenLem}
If $d(n)$ is a prime 
then $\langle\alpha_n,\beta_n\rangle=\Alt (d(n))$. 
\end{lem}

This is based on a classical fact from the theory 
of finite permutation groups. 

\begin{thm}[Jordan] \label{JordanThm}
Let $G \leq \Sym(d)$ be a primitive permutation group 
of degree $d$. Suppose that $G$ contains 
a $p$-cycle, for some prime $p\leq d-2$. 
Then $\Alt(d) \leq G$. 
\end{thm}

\begin{proof}[Proof of Lemma \ref{GenLem}]
Since $d(n)$ is odd, 
$\langle \alpha_n,\beta_n \rangle \leq \Alt (d(n))$. 
Clearly $\langle \alpha_n,\beta_n \rangle$ is 
transitive, and since $d(n)$ is prime, 
it is also primitive. 
The conclusion follows from Theorem \ref{JordanThm}. 
\end{proof}

Let $W = C_3 \wr \mathbb{Z} = \bigoplus_{\mathbb{Z}} C_3 \rtimes \mathbb{Z}$ be the regular restricted wreath product 
of $C_3$ and $\mathbb{Z}$ 
(so that $\mathbb{Z}$ acts on $\bigoplus_{\mathbb{Z}} C_3$ 
by shifting co-ordinates). 
Let $\alpha_{\infty}$ be a generator for $\mathbb{Z}$ 
and let $\beta_{\infty} \in \bigoplus_{\mathbb{Z}} C_3$ 
be a generator for the copy of $C_3$ supported at $0$, 
so that $\lbrace \alpha_{\infty} , \beta_{\infty} \rbrace$ 
is a generating set for $W$. 
A key ingredient in our results is the observation 
that, under mild hypotheses, 
the Cayley graphs of the $\Alt (d(n))$, 
with respect to $\lbrace \alpha_n,\beta_n \rbrace$ 
are locally isomorphic to that of 
$W$, with respect to 
$\lbrace \alpha_{\infty},\beta_{\infty} \rbrace$. 

\begin{propn} \label{BigFactorsProp}
Let $w (x,y) \in F(x,y)$ be a freely reduced word of 
length at most $n$ in the variables $x$ and $y$. 
Let $m \in \mathbb{N}$ and suppose: 
\begin{equation} \label{TripodSpreadEqn}
r_1 (m) , r_2 (m) , d(m) - r_1(m) - r_2 (m) \geq 2n+1. 
\end{equation}
Then $w (\alpha_m,\beta_m) = e$ in $\Alt(d(m))$, 
if and only if 
$w (\alpha_{\infty},\beta_{\infty}) = e$ in $W$. 
\end{propn}

\begin{proof}
Working from the left and applying 
relations of the form $x^i y^{\pm 1} = (x^i y x^{-i})^{\pm 1}x^i$, 
we may rewrite $w(x,y)$ in the form: 
\begin{equation*}
(x^{i_1} y x^{-i_1})^{\epsilon_1}\cdots(x^{i_1} y x^{-i_k})^{\epsilon_1} x^l
\end{equation*} 
for some $i_j , l \in \mathbb{Z}$ with 
$\lvert i_j \rvert , \lvert l \rvert \leq n$, 
and $\epsilon_j = \pm 1$. 
Letting $H=\langle x^i y x^{-i}
:\lvert i\rvert\leq n \rangle$, 
there exist $v(x,y) \in [H,H]$ 
and $c_i , l \in \mathbb{Z}$ such that: 
\begin{equation*}
w (x,y) = v(x,y) \Big(  \prod_{i=-n} ^n (x^i y x^{-i})^{c_i} \Big) x^l \text{.}
\end{equation*}
Now consider the elements: 
\begin{center}
$\alpha_m ^i \beta_m \alpha_m ^{-i} 
= \big( (i+1) \; (r_1(m)+i+1) \; (r_1(m)+r_2(m)+i+1) \big) \in \Alt (d(m))$ 
\end{center}
for $\lvert i \rvert \leq n$. 
By the inequalities (\ref{TripodSpreadEqn}) 
these $3$-cycles are disjointly supported, 
hence commute with one another. 
Meanwhile all $\alpha_{\infty} ^i \beta_{\infty} \alpha_{\infty} ^{-i}$ 
lie in $\bigoplus_{\mathbb{Z}} C_3$, 
an abelian group of exponent $3$. 
Thus $v(\alpha_m,\beta_m) = e$; 
$v(\alpha_{\infty},\beta_{\infty}) = e$ 
and $w(\alpha_m,\beta_m) \alpha_m ^{-l}$ 
is a permutation of order dividing $3$. 

Now suppose 
$w(\alpha_m,\beta_m) = e$. 
If $l \neq 0$, 
then the order of $\alpha_m$ is 
$d(m) \leq 3 \lvert l \rvert \leq 3n$, 
contradicting (\ref{TripodSpreadEqn}). 
Thus $l=0$, so that: 
\begin{equation} \label{3cyceqn}
w(\alpha_m,\beta_m) 
= \prod_{i=-n} ^n (\alpha_m ^i \beta_m \alpha_m ^{-i})^{c_i} 
\text{ and } 
w(\alpha_{\infty},\beta_{\infty}) 
= \prod_{i=-n} ^n (\alpha_{\infty} ^i \beta_{\infty} \alpha_{\infty} ^{-i})^{c_i}\text{.}
\end{equation} 
Since the $\alpha_m ^i \beta_m \alpha_m ^{-i}$ 
are disjointly supported $3$-cycles, 
$w(\alpha_m,\beta_m) = e$ and (\ref{3cyceqn}) 
together imply that $3 | c_i$ for all $i$. 
But then (\ref{3cyceqn}) yields 
$w(\alpha_{\infty},\beta_{\infty}) =e$ also. 

Conversely suppose $w(\alpha_{\infty},\beta_{\infty}) =e$. 
Since $\alpha_{\infty}$ has infinite order in $W$, 
we have $l=0$, so (\ref{3cyceqn}) holds again. 
Since 
$\alpha_{\infty} ^i \beta_{\infty}\alpha_{\infty} ^{-i} 
 \in \bigoplus_{\mathbb{Z}} C_3$
is a generator for the copy of $C_3$ supported at $i$, 
(\ref{3cyceqn}) implies that $3 | c_i$ for all $i$, 
so that $w(\alpha_m,\beta_m) = e$ also. 
\end{proof}

Let $\alpha = (\alpha_n) , \beta = (\beta_n) \in \prod_n \Alt(d(n))$ and let 
$S = \lbrace \alpha,\beta \rbrace$. 
We define the \emph{$(d,r_1,r_2)$-generalized B.H. Neumann group} to be $B(d,r_1,r_2)=\langle S \rangle\leq\prod_n \Alt(d(n))$. 
The group-theoretic properties of $B(d,r_1,r_2)$ depend 
strongly on the functions $d, r_1$ and $r_2$. 
If $d$ is bounded, then clearly $B(d,r_1,r_2)$ is a finite 
group. On the other hand, 
if $d$ is increasing and $r_1,r_2 \equiv 1$, 
then $B(d,r_1,r_2)$ is an infinite group, 
whose isomorphism-type uniquely determines the 
function $d$; these were the original groups 
studied bu B.H. Neumann \cite{Neum}. 
The groups of fast residual finiteness growth 
constructed in \cite{BoRaSewa} are all generalized B.H. Neumann groups. 

\section{Proofs of main results}

\subsection{Construction of the examples} \label{ConstrSubsect}


Let $\epsilon > 0$, let $C_0 > 0$ be as in Proposition \ref{BuildingrProp}, 
and let $C \geq C_0$. 
Throughout this Section, $f : \mathbb{N} \rightarrow \mathbb{N}_{\geq 3}$ 
is a nondecreasing function satisfying: 
\begin{equation} \label{MainSect3Ineq}
f(n) \geq C n \log (n) \big( \log \log(n) \big)^{1+\frac{\epsilon}{2}} + C
\end{equation}
for all $n\in\mathbb{N}$. 

\begin{ex} \label{MainFunctionEx}
Let $F : \mathbb{N} \rightarrow \mathbb{N}$ 
be a nondecreasing function satisfying 
condition (a) of Theorem \ref{MainThmRFG}, 
with respect to $c,\epsilon > 0$. 
Let $C_0$ satisfy
Proposition \ref{BuildingrProp} 
with respect to $\epsilon/2$ and let $C_1 = C \geq C_0$. 
Let $C_2 > 0$ satisfy the conclusion of Lemma \ref{FactorialInvLem} 
and set $f(n) = \lceil \log F(n+C_2) / \log\log F(n+C_2) \rceil$. 
By Lemma \ref{FactorialInvLem}, $f$ satisfies (\ref{MainSect3Ineq}). 
\end{ex}

We let $d : \mathbb{N}\rightarrow \mathbb{N}$ 
be a function satisfying the conclusion 
of Corollary \ref{PrimeFnCoroll} with respect to $f$ 
and let $q : \mathbb{N}\rightarrow \mathbb{N}$ 
be a nondecreasing function satisfying 
$n \leq q(n) \leq d(n)/4$ 
for all $n \in \mathbb{N}$. 
Then $d$ and $q$ satisfy the hypotheses 
of Proposition \ref{BuildingrProp}. 
Let $r : \mathbb{N}\rightarrow \mathbb{N}$ 
satisfy conclusions (a) and (b) 
of Proposition \ref{BuildingrProp}. 
Set $r_1 = r_2 = r$ and let 
$\alpha_n,\beta_n \in \Alt (d(m))$ be as in 
Subsection \ref{NeumannSubsect}. 


\begin{lem} \label{CommutingPermLem}
For $m,n \in \mathbb{N}$, 
$\beta_m$ and $\alpha_m ^{r(n)} \beta_m \alpha_m ^{-r(n)}$ commute in $\Alt (d(m))$ iff $m \neq n$. 
\end{lem}

\begin{proof}
Writing indices modulo $d(m)$, we have: 
\begin{equation*}
\alpha_m ^{r(n)} \beta_m \alpha_m ^{-r(n)} = \big((1+r(n)) \; (1+r(m)+r(n)) \; (1+2r(m)+r(n))\big)
\text{.}
\end{equation*}
Thus, when $r(n) \nequiv \pm r(m),\pm 2 r(m) \mod d(m)$, 
$\beta_m$ and $\alpha_m ^{r(n)} \beta_m \alpha_m ^{-r(n)}$ are either equal or disjointly supported, 
hence commute. 
By Proposition \ref{BuildingrProp}, 
this holds for all $m \neq n$. 
By contrast, 
$\beta_m \alpha_m ^{r(m)} \beta_m \alpha_m ^{-r(m)}$
sends $1$ to $1+r(m)$, whereas 
$\alpha_m ^{r(m)} \beta_m \alpha_m ^{-r(m)} \beta_m$ 
sends $1$ to $1 + 2 r(m) \nequiv 1 + r(m) \mod d(m)$, 
since $0 < r(m) < d(m)$. 
\end{proof}

Let $G = B(d,r,r)$ be the associated generalized 
B.H. Neumann group, 
generated by $S = \lbrace \alpha,\beta \rbrace$, 
as defined in Subsection \ref{NeumannSubsect}. 
Our goal for the remainder of the paper 
shall be to show that 
$\mathcal{F}_G ^S (n)$ is approximately 
$\lvert \Alt(d(n)) \rvert$; 
$\mathcal{R}_G ^S (n)$ is approximately 
$\prod_{m=1} ^n \lvert \Alt(d(m)) \rvert$, 
and $\mathcal{D}_G ^S (n)$ is approximately $d(n)$, 
from which Theorems \ref{MainThmRFG}, 
\ref{MainThmFullRFG} and \ref{MainThmDiv} shall follow. 

\subsection{Bounds on residual finiteness growth}

For $m \in \mathbb{N}$, let 
$\pi_m : \prod_k \Alt((d(k))) \rightarrow \Alt(d(m))$ 
be projection to the $m$th factor. 
Let $T_m \leq \prod_n \Alt(d(n))$ be the copy of 
$\Alt (d(m))$ supported at the $m$th co-ordinate; 
that is, 
\begin{equation*}
T_m = \bigcap_{m \neq n} \ker (\pi_n)
\end{equation*}
Thus for $w \in F(x,y)$ a freely reduced word 
and $m \in \mathbb{N}$, 
\begin{equation*}
\pi_m (w(\alpha,\beta)) = w(\alpha_m,\beta_m)\text{.}
\end{equation*}

\begin{lem} \label{SingleFactorEltLem}
For all $m \in \mathbb{N}$, 
$B_S \big( 4+4r(m) \big) \cap T_m \neq \lbrace e \rbrace$. 
\end{lem}

\begin{proof}
The element $[\beta,\alpha ^{r(m)} \beta \alpha ^{-r(m)}]$ of $G$ lies in $T_m$ 
by Lemma \ref{CommutingPermLem}, 
and clearly has word-length at most $4+4r(m)$ 
with respect to $S$. 
\end{proof}

\begin{coroll} \label{SimpleSubLem}
For all $m \in \mathbb{N}$, $T_m \leq G$. 
\end{coroll}

\begin{proof}
Since $T_m \vartriangleleft \prod_n \Alt(d(n))$, 
$G \cap T_m \vartriangleleft G$. 
For $t \in G \cap T_m$ and 
any word $w(x,y)$ in the variables $x$ and $y$,  
\begin{equation*}
w(\alpha_m,\beta_m) \pi_m(t) w(\alpha_m,\beta_m)^{-1} 
= \pi_m (w(\alpha,\beta) t w(\alpha,\beta)^{-1}) 
\in \pi_m (G \cap T_m)
\end{equation*}
so by Lemma \ref{GenLem}, 
$\pi_m (G \cap T_m) \vartriangleleft \Alt(d(m))$. 
By Lemma \ref{SingleFactorEltLem}, 
$G \cap T_m$ is nontrivial, 
but $\pi_m|_{T_m}$ is injective 
and $\Alt(d(m))$ is simple, 
so $\pi_m (G \cap T_m) = \Alt(d(m))$. 
The claim follows, since $T_m \cong \Alt(d(m))$. 
\end{proof}

\begin{propn} \label{RFGLBProp}
For all $m \in \mathbb{N}$, 
$\mathcal{F}_G ^S (4+4r(m)) \geq d(m)! / 2$ 
and $\mathcal{D}_G ^S (4+4r(m)) \geq d(m)$. 
\end{propn}

\begin{proof}
Let $e \neq g_m \in B_S \big( 4+4r(m) \big) \cap T_m$ 
(such an element exists 
by Lemma \ref{SingleFactorEltLem}). 
Let $Q$ be a finite group with 
$\lvert Q \rvert < d(m)! / 2$, 
and let $\phi : G \rightarrow Q$ be a homomorphism. 
By Corollary \ref{SimpleSubLem} we have 
an induced homomorphism $\phi |_{T_m} : T_m \rightarrow Q$, which is not injective, 
and by simplicity of $T_m$, 
must therefore be the trivial homomorphism. 
Thus $T_m \leq \ker (\phi)$, 
and in particular, $g_m \in \ker (\phi)$. 
Since this holds for all such $\phi$, 
$g_m$ has normal depth at least $d(m)! / 2$. 
If $H$ is a subgroup of $G$ of index $d<d(m)$, 
then the action of $G$ on the cosets of $H$ in $G$ 
induces a homomorphism $\phi : G \rightarrow \Sym(d)$ 
whose kernel is contained in $H$. 
Since $\lvert \Sym(d) \rvert < d(m)!/2$ 
we may argue as above, and deduce $g_m \in \ker(\phi) \leq H$, 
so that $g_m$ has depth at least $d(m)$. 
\end{proof}

\begin{propn} \label{RFGUBProp}
For all $n \in \mathbb{N}$, $\mathcal{F}_G ^S (n) \leq d(n)!/2$ 
and $\mathcal{D}_G ^S (n) \leq d(n)$. 
\end{propn}

\begin{proof}
Let $g \in B_S (n)$ and let $w (x,y)$ be a freely reduced word of 
length at most $n$ such that $g = w(\alpha,\beta)$ 
in $G$, so that $\pi_m (g) = w(\alpha_m,\beta_m)$. 
If $g \neq e$ in $G$, then there exists 
$m \in \mathbb{N}$ such that $\pi_m (g) \neq e$. 
By Proposition \ref{BigFactorsProp} 
we may assume $m \leq n$. 
Thus the normal depth of $g$ is at most $\lvert \Alt(d(m)) \rvert 
\leq d(n)!/2$ (since $d$ is increasing). 
Similarly, the depth of $g$ is at most $d(n)$, 
since for any $m \leq n$ for which $\pi_m (g) \neq e$, 
there is a point in $\lbrace 1,\ldots,d(m) \rbrace$ 
which is not fixed by $\pi_m (g)$. 
The preimage under $\pi_m$ of the stabiliser of this point 
is a subgroup of $G$, of index at most $d(m)$, 
in which $g$ does not lie. 
\end{proof}

\begin{propn} \label{FullUBProp}
For all $n \in \mathbb{N}$, 
\begin{equation*}
\mathcal{R}_G ^S (n) \leq \frac{1}{2^{2n}}
\prod_{k=1} ^{2n} d(k)!
\end{equation*}
\end{propn}

\begin{proof}
Define: 
\begin{equation*}
\rho_n = \times_{k=1} ^{2n} (\pi_k |_{G}) : G \rightarrow \prod_{k=1} ^{2n} \Alt (d(k))
\end{equation*}
and claim that $\rho_n|_{B_S (n)}$ is injective. 
For if not, then there exist distinct elements 
$g,h \in B_S (n)$ such that $\rho_n(g)=\rho_n(h)$. 
Let $w(x,y)$ be a freely reduced word 
of length at most $2n$ such that 
$w(\alpha,\beta) = g h^{-1} \neq e$ in $G$. 
There exists $m \in \mathbb{N}$ such that 
$\pi_m (g h^{-1}) = w(\alpha_m,\beta_m)\neq e$ 
in $\Alt (d(m))$. 
By Proposition \ref{BigFactorsProp}, 
we can assume $m \leq n$. 
Thus in any case, $\rho_n (gh^{-1})\neq e$, 
contradiction. 
\end{proof}

We now put everything together, 
and prove our most general bounds on 
the divisibility function, residual finiteness growth 
and full residual finiteness growth. 

\begin{thm} \label{MainTechDiv}
Suppose that $f : \mathbb{N} \rightarrow \mathbb{N}$ is nondecreasing and that 
there exists $c ,\epsilon > 0$  such that for all $n \in \mathbb{N}$, 
\begin{equation} \label{DivHypIneq}
f(n) \geq c n \log (n) \log \log (n)^{1+\epsilon}. 
\end{equation}
Then there exists a residually finite group $G$, 
generated by a finite set $S$, 
and $c_1, c_2 > 0$ such that: 
\begin{equation}
f(n/c_1 - c_2) \leq \mathcal{D}_G ^S (n) \leq 2 f (c_1 n + c_2)
\end{equation}
for all $n \in \mathbb{N}$. 
\end{thm}

\begin{proof}
First, replacing $f(n)$ by $f(n+C)$, for $C>0$ a sufficiently large constant, 
we may assume $f$ satisfies (\ref{MainSect3Ineq}) 
(note that this substitution does not affect the conclusion of the Theorem). 
Let $q(n) = n$ and let $d$, $r$, $G=B(d,r,r)$ 
and $S \subseteq G$ 
be as described in Subsection \ref{ConstrSubsect}. 
Note that our choice of $q$ is valid by hypothesis on $f$. 
We also have $r(n) < 18 n$. 
By Corollary \ref{PrimeFnCoroll} and 
Propositions \ref{RFGLBProp} and \ref{RFGUBProp} we have: 
\begin{equation*}
f(n) \leq d(n) \leq \mathcal{D}_G ^S (72n + 4) \leq d(72n + 4) \leq 2f(72n + 4)
\end{equation*}
for all $n \in \mathbb{N}$. 
The result follows. 
\end{proof}

The next two results apply to any nondecreasing function $F$ 
satisfying condition (a) of Theorem \ref{MainThmRFG}, 
namely: there exist $c,\epsilon > 0$ such that
\begin{equation} \label{Hypothesis(a)Recall}
F(n) \geq \exp \big( c n \log (n)^2 \log \log (n)^{1+\epsilon} \big)
\end{equation}
for all $n \in \mathbb{N}$. 

\begin{thm} \label{MainTechRFG}
Let $F : \mathbb{N} \rightarrow \mathbb{N}$ 
be a nondecreasing function 
satisfying (\ref{Hypothesis(a)Recall}). 
Then there exists a residually finite group $G$, 
generated by a finite set $S$, 
and $c_1 , c_2 , c_3 >1$ such that: 
\begin{equation} \label{MainTechRFGIneq}
F(n/c_1 - c_2)^{1-c_3 \frac{\log\log\log F(n)}{\log\log F(n)}} \leq \mathcal{F}_G ^S (n) \leq F(c_1 n + c_2)^{2+c_3 \frac{\log\log\log F(n)}{\log\log F(n)}}
\end{equation}
for all $n \in \mathbb{N}$. 
\end{thm}

\begin{proof}
Let $f$ be as in Example \ref{MainFunctionEx}, 
let $q(n) = n$ and let $d$, $r$, $G=B(d,r,r)$ 
and $S \subseteq G$ 
be as described in Subsection \ref{ConstrSubsect}. 
Note that our choice of $q$ is valid by 
hypothesis on $F$ and implies that $r(n) < 18 n$. 
By Corollary \ref{PrimeFnCoroll} and 
Propositions \ref{RFGLBProp} and \ref{RFGUBProp} we have: 
\begin{equation*}
\mathcal{F}_G ^S (n) \leq d(n)!/2 \leq (2f(n))! 
\text{ and }\mathcal{F}_G ^S (72n+4) \geq d(n)!/2 \geq f(n)!/2
\end{equation*} 
for all $n$. 
Since $f(n) = \lceil \log F(n+C_2) / \log\log F(n+C_2) \rceil$, and by Lemma \ref{FactorialBdLem} (a), 
$(2f(n))!$ and $f(n)!/2$ may be bounded above 
and below by: 
\begin{equation*}
F(n+C_2) ^{2 + O(\frac{\log\log\log F(n+C_2)}{\log\log F(n+C_2)})} \text{ and }F(n+C_2) ^{1 + O(\frac{\log\log\log F(n+C_2)}{\log\log F(n+C_2)})}
\end{equation*}
respectively. Thus the bounds (\ref{MainTechRFGIneq}) 
hold for all $n$ sufficiently large, 
and hence (enlarging the constants involved) 
for all $n$. 
\end{proof}

\begin{thm} \label{MainTechFRFG}
Let $F : \mathbb{N} \rightarrow \mathbb{N}$ 
be a nondecreasing function 
satisfying (\ref{Hypothesis(a)Recall}). 
Then there exists a residually finite group 
$G$, generated by a finite set $S$, 
and $c_1 , c_2 , c_3 >1$ such that: 
\begin{equation} \label{MainTechFRFGIneq}
F(n/c_1 - c_2)^{1-c_3 \frac{\log\log\log F(n)}{\log\log F(n)}} \leq \mathcal{R}_G ^S (n) 
\leq F(c_1 n + c_2)^{(2+c_3 \frac{\log(n)}{\log\log F(n)}) n}
\end{equation}
for all $n \in \mathbb{N}$. 
\end{thm}

\begin{proof}
Let $q$, $f$, $d$, $r$, $G$ and $S$ 
be as in the proof of Theorem \ref{MainTechRFG}. 
The lower bound follows from 
Lemma \ref{TrivialRFGrowthIneq} (i) and 
the conclusion of Theorem \ref{MainTechRFG}. 
For the upper bound, by 
Corollary \ref{PrimeFnCoroll} and Proposition \ref{FullUBProp} we have: 
\begin{align*}
\mathcal{R}_G ^S (\lfloor n/2 \rfloor) \leq \frac{1}{2^{n}}
\prod_{k=1} ^{n} d(k)! 
& \leq \big( d(1) + \cdots + d(n) \big)! \\
& \leq \big( 2f(1) + \cdots + 2f(n) \big)! \\
& \leq (2nf(n))!
\end{align*}
since $f$ is nondecreasing. 
The desired upper bound then follows from 
Lemma \ref{FactorialBdLem} (b). 
\end{proof}

\begin{rmrk} \normalfont
If we were to use a more sophisticated 
number-theoretic input than Bertrand's postulate 
in defining the function $d$, 
then it would be possible to reduce 
``$2$'' to ``$1$'' in the exponent in the upper bound 
for Theorems \ref{MainTechRFG} and \ref{MainTechFRFG}. 
For instance, using the work of Dusart \cite{Dusa}, 
we may replace the upper bound for $d$ in 
Corollary \ref{PrimeFnCoroll} with a bound: 
\begin{equation*}
d(n) \leq f(n) \big( 1 + \frac{1}{\log(f(n))^3} \big)
\end{equation*}
for all $n$ sufficiently large. 
\end{rmrk}

Finally, we complete the proof 
of our main results Theorem \ref{MainThmRFG}; 
Theorem \ref{MainThmFullRFG} and Theorem \ref{MainThmDiv}, 
which we restate here. 

\begin{thm} \label{MainThmRFGrestated} 
Let $F : \mathbb{N} \rightarrow \mathbb{N}$ 
be a nondecreasing function 
satisfying (\ref{Hypothesis(a)Recall}) 
and suppose there exist $C_1, C_2 > 1$ 
such that for all $n \in \mathbb{N}$, 
$F(n)^{C_1} \leq F(C_2 n)$. 
Then there exists a residually finite group 
$G$, generated by a finite set $S$, 
such that $\mathcal{F}_G ^S \approx^s F$. 
\end{thm}

\begin{proof}
Since $F$ is unbounded nondecreasing, 
we deduce from (\ref{MainTechRFGIneq}) that: 
\begin{equation} \label{FinalRFGIneq}
F(n/2 c_1)^{\frac{1}{2}} 
\leq \mathcal{F}_G ^S (n) 
\leq F(2 c_1 n)^{\frac{5}{2}} 
\end{equation}
for all $n$ sufficiently large. 
Repeatedly applying the inequality 
$F(n)^{C_1} \leq F(C_2 n)$ yields 
$F(n)^{C_1 ^r} \leq F(C_2 ^r n)$ 
and $F(n/C_2 ^r) \leq F(n)^{1/C_1 ^r}$ 
for all $r$ and all $n$ sufficiently large. 
Combining with (\ref{FinalRFGIneq}), 
we have (\ref{MainThmConclIneq}) for all $n$ 
sufficiently large, 
and the required result follows from 
Remark \ref{SuffLargeRmrk}. 
\end{proof}

\begin{thm} \label{MainThmFullRFGrestated}
Let $F : \mathbb{N} \rightarrow \mathbb{N}$ 
be a nondecreasing function 
satisfying (\ref{Hypothesis(a)Recall}) 
and suppose there exists $C > 1$ 
such that for all $n \in \mathbb{N}$, 
$F(n)^n \leq F(Cn)$. 
Then there exists a residually finite group 
$G$, generated by a finite set $S$, 
such that $\mathcal{R}_G ^S \approx^s F$. 
\end{thm}

\begin{proof}
By the upper bound in (\ref{MainTechFRFGIneq}) and our hypothesis on $F$, 
we have: 
\begin{equation}
\mathcal{R}_G ^S (n) 
\leq F(C(c_1 n + c_2))^{2+c_3 \frac{\log(n)}{\log\log F(n)}}
\end{equation}
for all $n \in \mathbb{N}$. 
As noted in the Introduction following Theorem \ref{MainThmFullRFG}, 
$F$ satisfies the conditions of Theorem \ref{MainThmRFGrestated}, 
and there exists $c > 0$ such that for all $n \in \mathbb{N}$, 
$F(n) \geq \exp (cn^{\log(n)})$, so that $\log\log F(n) \geq \log(n)^2 + \log(c)$. 
Thus for all $n$ sufficiently large, 
(\ref{FinalRFGIneq}) is satisfied, 
and arguing as in the proof of Theorem \ref{MainThmRFGrestated} 
we have the desired result. 
\end{proof}

\begin{thm} \label{MainThmDivrestated}
Let $f : \mathbb{N} \rightarrow \mathbb{N}$ 
be a nondecreasing function satisfying (\ref{DivHypIneq}), 
and suppose that there exist $C_1 , C_2 > 1$ 
such that for all $n \in \mathbb{N}$, 
$C_1 F(n) \leq F(C_2 n)$. 
Then there exists a residually finite group 
$G$, generated by a finite set $S$, 
such that $\mathcal{D}_G ^S \approx^s f$. 
\end{thm}

\begin{proof}
By Theorem \ref{MainTechDiv} we have: 
\begin{equation*}
f(n/2c_1) \leq \mathcal{D}_G ^S (n) \leq 2f(2c_1 n)
\end{equation*}
for all $n$ sufficiently large. 
Applying (b'') repeatedly, the right-hand side can in turn be bounded 
above by $f(c_3 n)$ for some larger constant $c_3$, 
and all $n$ sufficiently large. 
By Remark \ref{SuffLargeRmrk} we are done. 
\end{proof}

\begin{rmrk}
\normalfont
Theorem \ref{MainThmDivrestated} also implies a result for systolic growth 
(for sufficiently fast growth-types): 
Lemma \ref{TrivialRFGrowthIneq} (ii) and (iv) yield that: 
\begin{equation*}
\mathcal{D}_{G} ^S (n) \leq \Sigma_{G} ^S (n) 
\leq \mathcal{D}_{G} ^S (2n) ^{\lvert B_S (n) \rvert^2} 
\leq \mathcal{D}_{G} ^S (2n)^{\exp(Cn)}
\end{equation*}
(using the fact that every finitely generated group has at most exponential 
word growth). 
These upper and lower bounds for $\Sigma_{G} ^S (n)$ 
are comparably large when $\mathcal{D}_{G} ^S (n)$ grows 
at least like $\exp(\exp(n))$, 
and become strongly equivalent under an appropriate ``regularity'' hypothesis 
replacing hypothesis (b), (b') or (b'') from the Introduction. 
Of course, our result for any one of the four growth functions 
studied in this paper would imply results for the other three, 
using Lemma \ref{TrivialRFGrowthIneq}, 
albeit excluding far more ``slow'' growth-types 
than our more detailed analysis here. 
\end{rmrk}

\section{Final questions}

There are many interesting open questions about (various versions of) 
residual finiteness growth and related asymptotic invariants, 
see for instance the survey \cite{DerFerPen}. 
Regarding inverse problems, 
we can ask about the spectrum of possible ``slow'' residual finiteness growth types. 

\begin{qu}
For which nondecreasing functions $F:\mathbb{N}\rightarrow\mathbb{N}$ 
not satisfying condition (a) of Theorem \ref{MainThmRFG} 
does there exist a finitely generated residually finite group $G$ 
such that $\mathcal{F}_G \approx^s F$ 
(respectively $\mathcal{R}_G \approx^s F$). 
For instance: 
\begin{itemize}
\item[(i)] Does there exist a finitely generated residually finite group $G$ 
such that $\mathcal{F}_G$ (respectively $\mathcal{R}_G$) 
strongly dominates every polynomial but is strongly dominated by 
(and not strongly equivalent to) the exponential function? 

\item[(ii)] Suppose $\mathcal{F}_G$ 
is strongly dominated by some polynomial function. 
Does there exist $\alpha > 0$ such that 
$\mathcal{F}_G (n)$  is strongly equivalent to $n^{\alpha}$? 
Must $\alpha$ be an integer? 
\end{itemize}
\end{qu}

We expect that in (ii), $\alpha$ need not be an integer: 
conjecturally nonabelian free groups provide a counterexample; 
see \cite{BradThom}. That said, the analogous question to (ii) 
for full residual finiteness growth may well have a positive solution: 
by Gromov's Polynomial Growth Theorem, 
if $\mathcal{R}_G$ 
is strongly dominated by a polynomial function, 
then $G$ is virtually nilpotent. 
Full residual finiteness growth of nilpotent groups 
is studied in \cite{BouRabStud}, and integral polynomial growth is 
proven in certain cases. 

For $\mathcal{C}$ a class of finite groups, 
a group $G$ is \emph{residually $\mathcal{C}$} if, 
for every $e \neq g \in G$, there exists a group $Q \in \mathcal{C}$ 
and a surjective homomorphism $\pi : G \rightarrow Q$ 
such that $\pi(g) \neq e$. 
Similarly, $G$ is \emph{fully residually in $\mathcal{C}$} if, 
for every finite subset $F$ of $G$, 
there exists a group $Q \in \mathcal{C}$ 
and a surjective homomorphism $\pi : G \rightarrow Q$ 
whose restriction to $F$ is injective 
(note that the requirement of surjectivity may be omitted from the definition 
if the class $\mathcal{C}$ is subgroup-closed). 
For $G$ generated by the finite set $S$, 
one can define a \emph{residual $\mathcal{C}$ growth function} 
$\mathcal{F}_{G,\mathcal{C}} ^S$ 
(respectively, a \emph{full residual $\mathcal{C}$ growth function} 
$\mathcal{R}_{G,\mathcal{C}} ^S$) by taking the minima, 
in the definition of $\mathcal{F}_G ^S$ (respectively $\mathcal{R}_G ^S$), 
only over surjective homomorphisms $\pi : G \rightarrow Q$ 
and groups $Q$ lying in $\mathcal{C}$. 
The conclusion of our Theorem \ref{MainThmRFG} concerning $\mathcal{F}_G ^S$ 
in fact extends to $\mathcal{F}_{G,\mathcal{C}} ^S$, 
for any class $\mathcal{C}$ containing the finite alternating groups 
(since the finite quotients of $G$ constructed in the proof of the upper bound 
in Theorem \ref{MainThmRFG} are alternating groups). 
Presumably, a construction very similar to the one we have given here, 
but replacing the $\Alt(d(k))$ by finite groups of the form $\PSL_{d(k)}(p)$, say, 
would yield finitely generated groups 
with prescribed residual $\mathcal{C}$ growth, 
for $\mathcal{C}$ a class containing the finite simple groups of Lie type. 
In a different direction, the class of groups which are (fully) 
residually in the class $\mathcal{C}_p$ of finite $p$-groups 
(for some fixed prime $p$) has been much studied. 

\begin{qu}
Let $p$ be a prime. 
For which nondecreasing functions $F : \mathbb{N} \rightarrow \mathbb{N}$ 
does there exist a finitely generated residually $\mathcal{C}_p$ group $G$ 
such that $\mathcal{F}_{G,\mathcal{C}_p}$ 
(respectively $\mathcal{R}_{G,\mathcal{C}_p}$) is strongly equivalent to $F$? 
\end{qu}


\begin{thebibliography}{99}

\bibitem{BouRab} K. Bou-Rabee. 
{\it Quantifying residual finiteness.}
J. Algebra 323 (2010), 729--737. 

\bibitem{BoRaCheTim} K. Bou-Rabee, J. Chen, A. Timashova. 
{\it Residual finiteness growths of Lamplighter groups. }
arXiv:1909.03535 [math.GR]. 

\bibitem{BoRaMcR} K. Bou-Rabee, B. McReynolds, 
{\it Asymptotic growth and least common multiples in groups,} 
Bull. London Math. Soc. (2011) 43(6): 1059--1068. 

\bibitem{BoRaSewa} K. Bou-Rabee, B. Seward, 
{\it Arbitrarily large residual finiteness growth,} 
J. Reine Angwe. Math. (2016) Vol. 710, 199-204. 

\bibitem{BouRabStud} K. Bou-Rabee, D. Studenmund, 
{\it Full residual finiteness growths of nilpotent groups,} 
Israel Journal of Mathematics 214 (2016) 209--233. 

\bibitem{BradThom} H. Bradford, A. Thom, 
{\it Short laws for finite groups and residual finiteness growth,} 
Trans. Amer. Math. Soc. 371, no. 9 (2019), 6447--6462. 

\bibitem{DeLaHar} P. de la Harpe, 
{\it Topics in Geometric Group Theory, }
Chicago lectures in mathematics, University of Chicago Press 2000. 

\bibitem{DerFerPen} J. Der\'{e}, M. Ferov, M. Pengitore, 
{\it Survey on effective separability, }
arXiv:2201.13327 [math.GR]

\bibitem{Dusa} P. Dusart, 
{\it Explicit estimates of some functions over primes, }
The Ramanujan Journal 45 (2018), 227--251.

\bibitem{Neum} B.H. Neumann, 
{\it Some remarks on infinite groups, }
J. London Math. Soc. 12 (1937) 120--127. 

\end{thebibliography}
\end{document}